\documentclass[12pt,oneside]{article}

\usepackage{amsmath}
\usepackage{amssymb}
\usepackage{color}

\usepackage{caption,subcaption}

\author{Hui Zhou\footnote{Corresponding author.} \footnote{School of Mathematical Sciences, Peking University, Beijing, 100871, P.~R.~China. \newline \indent \indent E-mail addresses: zhouhpku17@pku.edu.cn, huizhou@math.pku.edu.cn, zhouhlzu06@126.com}, Qi Ding\footnote{School of Mathematics and Statistics, Lanzhou University, Lanzhou, Gansu, 730000, P.~R.~China; and \newline \indent \indent Yonyou Network Technology Co., Ltd., Yonyou Software Park, No. 68 Beiqing Road, Haidian District, Beijing, 100094, P.~R.~China. \newline \indent \indent E-mail addresses: dingqi0@yonyou.com, 892127976@qq.com}~ and Ruiling Jia\footnote{Pure Mathematics, Faculty of Science, The PLA Information Engineering University, Zhengzhou, Henan, 450001, P.~R.~China. \newline \indent \indent E-mail addresses: jiarl09@163.com, 452684687@qq.com}}

\title{On distance matrices of graphs}

\def\j{{\mathbf{j}}}
\def\D{{\sf D}}

\def\LapExp{{\sf LapExp}}
\def\cof{{\rm cof}}
\def\adj{{\rm adj}}
\def\sub{{\sf sub}}
\def\BI{{\sf BI}}
\def\bi{{\sf bi}}
\def\d{{\sf d}}
\def\w{{\sf w}}
\def\DMI{{\sf DMI}}
\def\m{{\sf m}}

\newtheorem{theorem}{Theorem}[section]
\newtheorem{corollary}[theorem]{Corollary}
\newtheorem{lemma}[theorem]{Lemma}
\newtheorem{remark}[theorem]{Remark}
\newtheorem{definition}[theorem]{Definition}

\newtheorem{question}[theorem]{Question}

\newcommand*{\QEDA}{\hfill\ensuremath{\blacksquare}}  
\newenvironment{proof}[1][\hspace{2ex}\textbf{\textit{Proof}.}\hspace{1ex}]{\begin{trivlist}\item[\hskip \labelsep {\bfseries #1}]}{\QEDA\end{trivlist}}

\begin{document}

\maketitle

\begin{abstract}
Distance well-defined graphs consist of connected undirected graphs, strongly connected directed graphs and strongly connected mixed graphs. Let $G$ be a distance well-defined graph, and let $\D(G)$ be the distance matrix of $G$. Graham, Hoffman and Hosoya~\cite{Graham DM of a directed graph} showed a very attractive theorem, expressing the determinant of $\D(G)$ explicitly as a function of blocks of $G$. In this paper, we study the inverse of $\D(G)$ and get an analogous theory, expressing the inverse of $\D(G)$ through the inverses of distance matrices of blocks of $G$ (see Theorem~\ref{thm inverse of distance matrix by modified LapExp matrix}) by the theory of Laplacian expressible matrices which was first defined by the first author~\cite{Zhou2017DMdistance-wellDefined}. A weighted cactoid digraph is a strongly connected directed graph whose blocks are weighted directed cycles. As an application of above theory, we give the determinant and the inverse of the distance matrix of a weighted cactoid digraph, which imply Graham and Pollak's formula and the inverse of the distance matrix of a tree.
\end{abstract}

\textbf{Keywords}: Distance matrix; Determinant; Inverse matrix; Weighted cactoid digraph.

\textbf{MSC}: 15A15, 05C05.


\section{Introduction}\label{section Introduction}


Let $G$ be a distance well-defined graph. We use $V(G)$ and $E(G)$ to denote the vertex set and edge (arc) set of $G$ respectively. The \emph{distance} $\partial_G(u,v)$ from vertex $u$ to vertex $v$ in $G$ is the length (number of (directed) edges) of the shortest (directed) path from $u$ to $v$ in $G$. The \emph{distance matrix} $\D(G)$ of $G$ is a $|V(G)|\times |V(G)|$ square matrix whose $(u,v)$-entry is the distance $\partial_G(u,v)$, that is $\D(G)=(\partial_G(u,v))_{u,v\in V(G)}$. A \emph{cut vertex} of $G$ is a vertex whose deletion results in a disconnected graph. A \emph{block} of $G$ is a distance well-defined subgraph on at least two vertices such that it has no cut vertices and is maximal with respect to this property.

Let $\j$ be an appropriate size column vector whose entries are ones, let $\mathbf{0}$ be an appropriate size matrix whose entries are zeroes, and let $I$ be the identity matrix with an appropriate size. For any matrix $M$ and any column vector $\alpha$, we use $M^T$ and $\alpha^T$ to denote their transposes, respectively.

A square matrix $L$ is called a \emph{Laplacian-like matrix} if $L\j=\mathbf{0}$ and $\j^TL=\mathbf{0}$. This definition was first defined by the first author~\cite{Zhou2017DMdistance-wellDefined} to generalize the Laplacian matrix of an undirected graph. To study the inverse of the distance matrix of a distance well-defined graph, the first author~\cite{Zhou2017DMdistance-wellDefined} gave the definition of \emph{Laplacian expressible matrices}. A square matrix $D$ is \emph{left (or right) Laplacian expressible} if there exist a number $\lambda$, a column vector $\beta$ with $\beta^T\j=1$, and a square matrix $L$ such that \begin{center}$L\j=\mathbf{0}$, $\beta^TD=\lambda\j^T$ and $LD+I=\beta\j^T$\\ (or $\j^TL=\mathbf{0}$, $D\beta=\lambda\j$ and $DL+I=\j\beta^T$).\end{center} If the number $\lambda\neq 0$, then the Laplacian expressible matrix $D$ is invertible with $D^{-1}=-L+\frac{1}{\lambda}\beta\beta^T$ (see Lemma~\ref{lem inv of a matrix by lap}) and the matrix $L$ is Laplacian-like (see Lemma~\ref{lem Laplacian-like matrix}), which means the inverse $D^{-1}$ is expressed as the sum of a Laplacian-like matrix and a rank one matrix. This is why ``Laplacian expressible" comes. The main theory of~\cite{Zhou2017DMdistance-wellDefined} shows that if the distance matrix of each block of a graph is left (or right) Laplacian expressible, then the distance matrix of the graph is also left (or right) Laplacian expressible. This theory is very helpful for us to calculate the inverse of the distance matrix of a graph whose blocks correspond to Laplacian expressible matrices.

In this paper, we investigate the inverse of the distance matrix of a graph further and give a generalization of the Laplacian expressible theory in~\cite{Zhou2017DMdistance-wellDefined}. We give the following crucial definition (see also Definition~\ref{def modified Laplacian expressible matrix}). A square matrix $D$ is \emph{modified left (or right) Laplacian expressible} if there exist a number $\lambda$, column vectors $\alpha$ and $\beta$, and a square matrix $L$ such that \begin{center}$\alpha^T\j=1$, $L\j=\mathbf{0}$, $\alpha^TD=\lambda\j^T$ and $LD+I=\beta\j^T$\\ (or $\j^T\beta=1$, $\j^TL=\mathbf{0}$, $D\beta=\lambda\j$ and $DL+I=\j\alpha^T$).\end{center} By definition, we know the Laplacian expressible matrix is a special case of the modified Laplacian expressible matrix (when $\alpha=\beta$). Let $D$ be a modified left (or right) Laplacian expressible matrix as above. Suppose $\lambda\neq 0$ and $\j^T\beta=1\text{ (or $\alpha^T\j=1$)}$, then by Lemmas~\ref{lem inv of a matrix by lap} and \ref{lem Laplacian-like matrix}, the matrix $D$ is invertible, $L$ is Laplacian-like, and the inverse $D^{-1}=-L+\frac{1}{\lambda}\beta\alpha^T$ is expressed as the sum of a Laplacian-like matrix and a rank one matrix.

If the distance matrices of blocks of a graph $G$ are modified left (or right) Laplacian expressible, then we get the distance matrix of $G$ is also modified left (or right) Laplacian expressible (see Theorem~\ref{thm inverse of distance matrix by modified LapExp matrix}). So this modified Laplacian expressible property of a graph can retain from its blocks. This is helpful for us to calculate the inverse of the distance matrix of a graph whose blocks have modified Laplacian expressible distance matrices. This theory generalizes results in~\cite{Zhou2017DMdistance-wellDefined}, which first used Laplacian expressible matrices and got the inverses of distance matrices of the following graphs: trees, weighted trees, block graphs, odd-cycle-clique graphs, bi-block graphs, cactoid digraphs and complete multipartite graphs under some condition, etc (references are therein). But our theory can do more, such as arc weighted trees~\cite{Bapat bidirected tree,ZhouDing2016DMweightedTree}, mixed block graphs~\cite{ZhouDing2017MixedBlockGraphs}, weighted cactoid digraphs (see Theorem~\ref{thm weighted cactoid digraph}), etc. Theorem~\ref{thm weighted cactoid digraph} implies the corresponding results of weighted trees~\cite{Bapat bidirected tree,ZhouDing2016DMweightedTree} and cactoid digraphs~\cite{Hou Chen Inverse of cactoid graph}. 
For each of the above graphs, the inverse of its distance matrix can be expressed as the sum of a Laplacian-like matrix and a rank one matrix.

\section{The modified Laplacian expressible matrix}\label{section modified Laplacian expressible matrix}

First, we give a formula on the inverse of a square matrix.

\begin{lemma}\label{lem inv of a matrix by lap}
Let $D$ be an $n\times n$ matrix. Let $\lambda$ be an nonzero number, let $\alpha$ and $\beta$ be $n\times 1$ column vectors, and let $L$ be an $n\times n$ matrix. If either
\begin{enumerate}
\item $\alpha^TD=\lambda\j^T$ and $LD+I=\beta\j^T$, or
\item $D\beta=\lambda\j$ and $DL+I=\j\alpha^T$,
\end{enumerate}
then $D$ is invertible and $D^{-1}=-L+\frac{1}{\lambda}\beta\alpha^T$.
\end{lemma}
\begin{proof}The proofs of the two cases are similar, so we only give the proof of the first case. Suppose $\alpha^TD=\lambda\j^T$ and $LD+I=\beta\j^T$. Then $\beta\alpha^TD=\lambda\beta\j^T$ and $\beta\j^T=\frac{1}{\lambda}\beta\alpha^TD$. So $LD+I=\frac{1}{\lambda}\beta\alpha^TD$ and $(-L+\frac{1}{\lambda}\beta\alpha^T)D=I$. Thus $D^{-1}=-L+\frac{1}{\lambda}\beta\alpha^T$.
\end{proof}

Now we give some sufficient conditions for a square matrix $L$ to be a Laplacian-like matrix.

\begin{lemma}\label{lem Laplacian-like matrix}
Let $D$ be an $n\times n$ invertible matrix. Let $\lambda$ be an nonzero number, let $\alpha$ and $\beta$ be $n\times 1$ column vectors, and let $L=-D^{-1}+\frac{1}{\lambda}\beta\alpha^T$. Then $L$ is a Laplacian-like matrix if one of the following conditions holds:
\begin{enumerate}
\item $\alpha^TD=\lambda\j^T$, $\j^T\beta=1$ and $L\j=\mathbf{0}$,
\item $D\beta=\lambda\j$, $\alpha^T\j=1$ and $\j^T L=\mathbf{0}$,
\item $\alpha^TD=\lambda\j^T$, $D\beta=\lambda\j$ and $\j^T\beta=\alpha^T\j=1$.\label{case Laplacian-like 3}
\end{enumerate}
\end{lemma}

\begin{proof}
The proofs are similar, so we only give the proof of the second case. Since $D\beta=\lambda\j$, we have $D^{-1}\j=\frac{1}{\lambda}\beta$. Thus $L\j=-D^{-1}\j+\frac{1}{\lambda}\beta\alpha^T\j=-\frac{1}{\lambda}\beta+\frac{1}{\lambda}\beta=\mathbf{0}$. Hence $L$ is a Laplacian-like matrix.
\end{proof}


For ease of reference, we repeat the following definition.

\begin{definition}[Modified Laplacian expressible matrix]\label{def modified Laplacian expressible matrix}
Let $D$ be an $n\times n$ matrix. Let $\lambda$ be a number, let $\alpha$ and $\beta$ be $n\times 1$ column vectors, and let $L$ be an $n\times n$ matrix. If \begin{center}$\alpha^T\j=1$, $L\j=\mathbf{0}$, $\alpha^TD=\lambda\j^T$ and $LD+I=\beta\j^T$,\end{center} then we call $D$ a modified left Laplacian expressible matrix, or \begin{center}a left $\LapExp^*(\lambda,\alpha,\beta,L)$ matrix\end{center} to specify the corresponding parameters $\lambda,\alpha,\beta,L$. If \begin{center}$\j^T\beta=1$, $\j^TL=\mathbf{0}$, $D\beta=\lambda\j$ and $DL+I=\j\alpha^T$,\end{center} then we call $D$ a modified right Laplacian expressible matrix, or \begin{center}a right $\LapExp^*(\lambda,\alpha,\beta,L)$ matrix\end{center} to specify the corresponding parameters $\lambda,\alpha,\beta,L$. The matrix $D$ is called a modified Laplacian expressible matrix, if either $D$ is a modified left Laplacian expressible matrix, or $D$ is a modified right Laplacian expressible matrix.
\end{definition}

\begin{remark}\label{remark modified LapExp matrix}
The above Definition~\ref{def modified Laplacian expressible matrix} is also suitable and the above lemmas in this section are also true if $\lambda$ and the entries of vectors and matrices are taken from a commutative ring with identity. In this case, the condition ``$\lambda\neq 0$" should be changed to ``$\lambda$ is invertible".
\end{remark}

Let $D$ be a left (or right) $\LapExp^*(\lambda,\alpha,\beta,L)$ matrix, and suppose $\lambda\neq 0$. Then by Lemma~\ref{lem inv of a matrix by lap}, the matrix $D$ is invertible and $D^{-1}=-L+\frac{1}{\lambda}\beta\alpha^T$. If we assume $\j^T\beta=1$ (or $\alpha^T\j=1$) additionally, then by Lemma~\ref{lem Laplacian-like matrix}, the matrix $L$ is a Laplacian-like matrix, and so the inverse matrix $D^{-1}$ can be expressed as the sum of a Laplacian-like matrix and a rank one matrix. This is why we call the matrix $D$ a modified Laplacian expressible matrix. 


\section{Inverse of the generalized distance matrix}\label{section Inverse of generalized distance matrix}


\begin{definition}[Generalized distance matrix]\label{def Generalised distance matrix}
Let $G$ be a distance well-defined graph. A generalized distance matrix $D$ of $G$ is a $|V(G)|\times |V(G)|$ matrix $(D_{uv})_{u,v\in V(G)}$ whose entries are taken from a commutative ring with identity and satisfy the following conditions:
\begin{enumerate}
\item $D_{uu}=0$ for all $u\in V(G)$, and
\item if $u$ and $v$ are two vertices of $G$ such that every shortest (directed) path from $u$ to $v$ passes through the cut-vertex $x$, then $D_{uv}=D_{ux}+D_{xv}$.
\end{enumerate}
\end{definition}

Let $G$ be a distance well-defined graph with $n$ vertices. The distance matrix of $G$ is actually a generalized distance matrix. Let $G_1,G_2,\ldots,G_r$ be all the blocks of $G$ where $r\geqslant 1$. For each $1\leqslant i\leqslant r$, let $n_i\geqslant 2$ be the number of vertices of $G_i$. We call $(G^n;G_1^{n_1},G_2^{n_2},\ldots,G_r^{n_r})$ the \emph{structure parameters} of $G$. Note that \begin{equation*}n-1=\sum\limits_{i=1}^r(n_i-1).\end{equation*} An \emph{$n$-bag} is a tuple $(A,\lambda,\alpha,\beta,L)$ consisting of $n\times n$ matrices $A$ and $L$, a number $\lambda$ and $n\times 1$ column vectors $\alpha$ and $\beta$. An $n$-bag $(A,\lambda,\alpha,\beta,L)$ is called a \emph{left (or right) $\LapExp^*$ $n$-bag} if $A$ is a left (or right) $\LapExp^*(\lambda,\alpha,\beta,L)$ matrix. Let $M$ be an $n\times n$ matrix whose rows and columns are indexed by vertices of $G$, and let $H$ be a subgraph of $G$. We use $\sub(M;G,H)$ to denote the submatrix of $M$ whose rows and columns are corresponding to the vertices of $H$. Let $v$ be a vertex of $G$. The \emph{block index set} $\BI_G(v)$ of $v$ is the set of all indices $k$, $1\leqslant k\leqslant r$, satisfying that $v$ is a vertex of the block $G_k$. The \emph{block index} of $v$ is the cardinality $\bi_G(v)=\#\BI_G(v)$. The \emph{block index} of $G$ is $\bi(G)=\sum\limits_{v\in V(G)}(\bi_G(v)-1)$. By Lemma~4.2 in~\cite{Zhou2017DMdistance-wellDefined}, we have \begin{equation}\label{eqn bi(G)=r-1} \bi(G)=r-1.\end{equation} So \begin{equation*}n+\bi(G)=\sum\limits_{i=1}^rn_i.\end{equation*}



\begin{definition}[Composition bag]\label{def composition of bags}
Let $G$ be a distance well-defined graph with structure parameters $(G^n;G_1^{n_1},G_2^{n_2},\ldots,G_r^{n_r})$. Let $D$ be a generalized distance matrix of $G$. For each $1\leqslant i\leqslant r$, let $D_i=\sub(D;G,G_i)$ and let $B_i=(D_i,\lambda_i,\alpha_i,\beta_i,L_i)$ be an $n_i$-bag. The composition bag of bags $B_1,B_2,\ldots,B_r$ is an $n$-bag $(D,\lambda,\alpha,\beta,L)$ whose parameters are defined as follows:
\begin{eqnarray*}
\lambda & = & \sum\limits_{i=1}^r\lambda_i,\\
\alpha_v & = & \sum\limits_{i\in \BI_G(v)}(\alpha_i)_v-\bi_G(v)+1,\\
\beta_v & = & \sum\limits_{i\in \BI_G(v)}(\beta_i)_v-\bi_G(v)+1,\\
L & = & \sum\limits_{i=1}^r\hat{L_i},
\end{eqnarray*}
where
\begin{enumerate}
\item for any vertex $v$ of $G$, the entry of $\alpha$ corresponding to $v$ is $\alpha_v$ and the entry of $\beta$ corresponding to $v$ is $\beta_v$, and
\item for each $1\leqslant i\leqslant r$, $\hat{L_i}$ is an $n\times n$ matrix such that the entry of $\hat{L_i}$ corresponding to vertices $u$ and $v$ of $G$ is \begin{center}$(\hat{L_i})_{uv}=\left\{\begin{array}{cl} (L_i)_{uv}, & \text{if }u,v\in V(G_i),\\ 0, & \text{otherwise}.\end{array}\right.$\end{center}
\end{enumerate}
\end{definition}

Note that in the above definition, for any vertex $v$ of $G$, we can write
\begin{eqnarray*}
\alpha_v-1 & = & \sum\limits_{i\in \BI_G(v)}\Bigl((\alpha_i)_v-1\Bigr),\\
\beta_v-1 & = & \sum\limits_{i\in \BI_G(v)}\Bigl((\beta_i)_v-1\Bigr);
\end{eqnarray*}
this means $\alpha_v$ (or $\beta_v$) depends only on entries $(\alpha_i)_v$ $\Bigl($or $(\beta_i)_v\Bigr)$ with $i\in \BI_G(v)$.

In general, we get the following theorem on the generalized distance matrix of a distance well-defined graph whose blocks are corresponding to (left or right) $\LapExp^*$ bags.


\begin{theorem}\label{thm inverse of distance matrix by modified LapExp matrix}
Let $G$ be a distance well-defined graph with structure parameters $(G^n;G_1^{n_1},G_2^{n_2},\ldots,G_r^{n_r})$. Let $D$ be a generalized distance matrix of $G$. For each $1\leqslant i\leqslant r$, let $D_i=\sub(D;G,G_i)$ and let $B_i=(D_i,\lambda_i,\alpha_i,\beta_i,L_i)$ be a left (or right) $\LapExp^*$ $n_i$-bag. Let $B=(D,\lambda,\alpha,\beta,L)$ be the composition bag of bags $B_1,B_2,\ldots,B_r$. Then \begin{center}$\alpha^T\j=1$, $L\j=\mathbf{0}$, $\alpha^TD=\lambda\j^T$ and $LD+I=\beta\j^T$\\ (or $\j^T\beta=1$, $\j^TL=\mathbf{0}$, $D\beta=\lambda\j$ and $DL+I=\j\alpha^T$).\end{center} As a consequence, $B=(D,\lambda,\alpha,\beta,L)$ is a left (or right) $\LapExp^*$ $n$-bag. Furthermore, if $\lambda$ is invertible, then $D$ is invertible and $D^{-1}=-L+\frac{1}{\lambda}\beta\alpha^T$.
\end{theorem}

\begin{proof}
The proofs of left and right cases are similar. Here we only give the proof of the left case. For each $1\leqslant i\leqslant r$, the bag $B_i=(D_i,\lambda_i,\alpha_i,\beta_i,L_i)$ is a left $\LapExp^*$ $n_i$-bag, i.e. \begin{center}$\alpha_i^T\j=1$, $L_i\j=\mathbf{0}$, $\alpha_i^TD_i=\lambda_i\j^T$ and $L_iD_i+I=\beta_i\j^T$.\end{center} By counting in two ways, we have \begin{center}$r=\sum\limits_{i=1}^r\alpha_i^T\j=\sum\limits_{i=1}^r\sum\limits_{v\in V(G_i)}(\alpha_i)_v=\sum\limits_{v\in V(G)}\sum\limits_{i\in \BI_G(v)}(\alpha_i)_v$.\end{center} Then by Equation~(\ref{eqn bi(G)=r-1}), we get
\begin{eqnarray*}
\alpha^T\j=\sum\limits_{v\in V(G)}\alpha_v&=&\sum\limits_{v\in V(G)}\left(\sum\limits_{i\in \BI_G(v)}(\alpha_i)_v-\bi_G(v)+1\right)\\
&=&\sum\limits_{v\in V(G)}\sum\limits_{i\in \BI_G(v)}(\alpha_i)_v-\sum\limits_{v\in V(G)}\Bigl(\bi_G(v)-1\Bigr)\\
&=&r-\bi(G)=1.
\end{eqnarray*}
For each $1\leqslant i\leqslant r$, since $L_i\j=\mathbf{0}$, we get $\hat{L_i}\j=\mathbf{0}$. By definition, we have $L\j=\sum\limits_{i=1}^r\hat{L_i}\j=\mathbf{0}$.

We will prove $\alpha^TD=\lambda\j^T$ by induction on the block number $r$. The case $r=1$ is clear. Now let $r \geqslant 2$. Without loss of generality, we may assume $H=G_r$ is a leaf block of $G$ with separating vertex $x_0$. Let $B_H=(D_H,\lambda_H,\alpha_H,\beta_H,L_H)=B_r$. Then $\alpha_H^T\j=1$ and $\alpha_H^TD_H=\lambda_H\j^T$. Let $F=G-(H-x_0)$. Then the blocks of $F$ are $G_1,G_2,\ldots,G_{r-1}$. Let $D_F=\sub(D;G,F)$. Then $D_F$ is a generalized distance matrix of $F$, and $D_i=\sub(D_F;F,G_i)$ for $1\leqslant i\leqslant r-1$. Let $B_F=(D_F,\lambda_F,\alpha_F,\beta_F,L_F)$ be the composition bag of $B_1,\ldots,B_{r-1}$. Then $\alpha_F^T\j=1$ by the first paragraph of this proof, and $\alpha_F^TD_F=\lambda_F\j^T$ by the induction hypothesis.

Assume the vertex set of $F$ and $H$ are $V(F)=\{x_0,x_1,x_2,\ldots,x_p\}$ and $V(H)=\{x_0,x_{p+1},x_{p+2},\ldots,x_{p+q}\}$, respectively. Then the vertex set of $G$ is $V(G)=\{x_0,x_1,x_2,\ldots,x_p,x_{p+1},\ldots,x_{p+q}\}$. The matrices $D_F=(d_{uv})_{u,v\in V(F)}$ and $D_H=(d_{uv})_{u,v\in V(H)}$ are generalized distance matrices of $F$ and $H$, respectively. We can write them in the following form
\begin{eqnarray*}
D_F & = & \left(
\begin{array}{cc}
0 & \begin{array}{ccc}
a_1 & \cdots &a_p
\end{array}\\
\begin{array}{c}
c_1\\
\vdots \\
c_p
\end{array} & D_{F,x_0}
\end{array}
\right)=\left(
\begin{array}{cc}
0 & a^T\\ c & D_{F,x_0}
\end{array}
\right),\\
D_H & = & \left(
\begin{array}{cc}
0 & \begin{array}{ccc}
b_1 & \cdots &b_q
\end{array}\\
\begin{array}{c}
d_1\\
\vdots \\
d_q
\end{array} & D_{H,x_0}
\end{array}
\right)=\left(
\begin{array}{cc}
0 & b^T\\ d & D_{H,x_0}
\end{array}
\right),
\end{eqnarray*}
where $a^T=[a_1,\cdots,a_p]$, $b^T=[b_1,\cdots,b_q]$, $c^T=[c_1,\cdots,c_p]$, $d^T=[d_1,\cdots,d_q]$, $D_{F,x_0}=(d_{x_ix_j})_{1\leqslant i,j\leqslant p}$ and $D_{H,x_0}=(d_{x_{p+i}x_{p+j}})_{1\leqslant i,j\leqslant q}$. Then by the definition of the generalized distance matrix, we have
\begin{equation*}
D=\left(
\begin{array}{ccc}
0 & a^T & b^T \\
c & D_{F,x_0} & M \\
d & N & D_{H,x_0}
\end{array}
\right),
\end{equation*}
where $M=(m_{ij})_{p\times q}$ with $m_{ij}=c_i+b_j$ for $1\leqslant i \leqslant p$ and $1\leqslant j \leqslant q$, and $N=(n_{st})_{q\times p}$ with $n_{st}=d_s+a_t$ for $1\leqslant s \leqslant q$ and $1\leqslant t \leqslant p$. Then $M=c\j^T+\j b^T$ and $N=d\j^T+\j a^T$. Note that
\begin{eqnarray*}
\BI_G(x_0) & = & \BI_F(x_0)\cup \BI_H(x_0);\\
\BI_G(x_i) & = & \BI_F(x_i), \text{ for $1\leqslant i\leqslant p$};\\
\BI_G(x_{p+j}) & = & \BI_H(x_{p+j}), \text{ for $1\leqslant j\leqslant q$}.
\end{eqnarray*}
Then
\begin{eqnarray*}
\alpha_{x_0} = (\alpha_F)_{x_0}+(\alpha_H)_{x_0}-1; && \beta_{x_0} = (\beta_F)_{x_0}+(\beta_H)_{x_0}-1;\\
\alpha_{x_i} = (\alpha_F)_{x_i}, && \beta_{x_i} = (\beta_F)_{x_i}, \text{ for $1\leqslant i\leqslant p$};\\
\alpha_{x_{p+j}} = (\alpha_H)_{x_{p+j}}, && \beta_{x_{p+j}} = (\beta_H)_{x_{p+j}}, \text{ for $1\leqslant j\leqslant q$}.
\end{eqnarray*}
Let
\begin{eqnarray*}
\hat{\alpha}_F^T=[\alpha_{x_1},\alpha_{x_2},\cdots,\alpha_{x_p}], && \hat{\beta}_F^T=[\beta_{x_1},\beta_{x_2},\cdots,\beta_{x_p}],\\
\hat{\alpha}_H^T=[\alpha_{x_{p+1}},\alpha_{x_{p+2}},\cdots,\alpha_{x_{p+q}}], && \hat{\beta}_H^T=[\beta_{x_{p+1}},\beta_{x_{p+2}},\cdots,\beta_{x_{p+q}}].
\end{eqnarray*}
So we can write
\begin{eqnarray*}
\alpha^T = [\alpha_{x_0},\hat{\alpha}_F^T,\hat{\alpha}_H^T], && \beta^T = [\beta_{x_0},\hat{\beta}_F^T,\hat{\beta}_H^T],\\
\alpha_F^T = [(\alpha_F)_{x_0},\hat{\alpha}_F^T], && \beta_F^T = [(\beta_F)_{x_0},\hat{\beta}_F^T],\\
\alpha_H^T = [(\alpha_H)_{x_0},\hat{\alpha}_H^T], && \beta_H^T = [(\beta_H)_{x_0},\hat{\beta}_H^T].
\end{eqnarray*}

Let $\gamma^T=\alpha^TD$. For the entry corresponding to the vertex $x_0$,
\begin{eqnarray*}\gamma_{x_0}=\alpha^TD[\cdot|x_0]&=&\hat{\alpha}_F^Tc+\hat{\alpha}_H^Td\\
&=&\alpha_F^TD_F[\cdot|x_0]+\alpha_H^TD_H[\cdot|x_0]\\
&=&(\alpha_F^TD_F)[\cdot|x_0]+(\alpha_H^TD_H)[\cdot|x_0]\\
&=&\lambda_F+\lambda_H=\lambda.
\end{eqnarray*}
Let $v$ be a vertex of $F-x_0$. Then $N[\cdot|v]=d+a_v\j$ and
\begin{eqnarray*}\gamma_v&=&\alpha^TD[\cdot|v]\\
&=&\alpha_{x_0}a_v+\hat{\alpha}_F^TD_{F,x_0}[\cdot|v]+\hat{\alpha}_H^TN[\cdot|v]\\
&=&\alpha_{x_0}a_v+\alpha_F^TD_F[\cdot|v]-(\alpha_F)_{x_0}a_v+\alpha_H^TD_H[\cdot|x_0]+a_v\hat{\alpha}_H^T\j\\
&=&a_v[\alpha_{x_0}-(\alpha_F)_{x_0}+1-(\alpha_H)_{x_0}]+(\alpha_F^TD_F)[\cdot|v]+(\alpha_H^TD_H)[\cdot|x_0]\\
&=&\lambda_F+\lambda_H=\lambda.
\end{eqnarray*}
Let $w$ be a vertex of $H-x_0$. A similar argument gives $\gamma_w=\lambda$. Now we have $M[\cdot|w]=c+b_w\j$. Then \begin{eqnarray*}\gamma_w&=&\alpha^TD[\cdot|w]\\
&=&\alpha_{x_0}b_w+\hat{\alpha}_F^TM[\cdot|w]+\hat{\alpha}_H^TD_{H,x_0}[\cdot|w]\\
&=&\alpha_{x_0}b_w+\alpha_F^TD_F[\cdot|x_0]+b_w\hat{\alpha}_F^T\j+\alpha_H^TD_H[\cdot|w]-(\alpha_H)_{x_0}b_w\\
&=&b_w[\alpha_{x_0}+1-(\alpha_F)_{x_0}-(\alpha_H)_{x_0}]+(\alpha_F^TD_F)[\cdot|x_0]+(\alpha_H^TD_H)[\cdot|w]\\
&=&\lambda_F+\lambda_H=\lambda.
\end{eqnarray*}
By the above calculation, we have $\gamma^T=\lambda\j^T$.

At the end of this proof, we will prove $LD+I=\beta\j^T$ by induction on the block number $r$. The case $r=1$ is clear. Now let $r \geqslant 2$, and we use notation as above. Then $L_H\j=\mathbf{0}$ and $L_HD_H+I=\beta_H\j^T$. For the subgraph $F$, we have $L_F\j=\mathbf{0}$ by the first paragraph of this proof, and $L_FD_F+I=\beta_F\j^T$ by the induction hypothesis. The matrices $L_F$ and $L_H$ can be written as
\begin{eqnarray*}
L_F & = & \left(
\begin{array}{cc}
\varphi & \begin{array}{ccc}
f_1 & \cdots & f_p
\end{array}\\
\begin{array}{c}
h_1\\
\vdots \\
h_p
\end{array} & L_{F,x_0}
\end{array}
\right)=\left(
\begin{array}{cc}
\varphi & f^T\\ h & L_{F,x_0}
\end{array}
\right),\\
L_H & = & \left(
\begin{array}{cc}
\tau & \begin{array}{ccc}
g_1 & \cdots & g_q
\end{array}\\
\begin{array}{c}
k_1\\
\vdots \\
k_q
\end{array} & L_{H,x_0}
\end{array}
\right)=\left(
\begin{array}{cc}
\tau & g^T\\ k & L_{H,x_0}
\end{array}
\right),
\end{eqnarray*}
where $f^T=[f_1,\cdots,f_p]$, $g^T=[g_1,\cdots,g_q]$, $h^T=[h_1,\cdots,h_p]$ and $k^T=[k_1,\cdots,k_q]$.
Then by the definition of $L$, we have
\begin{equation*}
L=\left(
\begin{array}{ccc}
\pi & f^T & g^T \\
h & L_{F,x_0} & Q \\
k & Q^T & L_{H,x_0}
\end{array}
\right),
\end{equation*}
where $\pi=\varphi+\tau$, and $Q_{p\times q}=\mathbf{0}$. Assume
\begin{eqnarray*}
L_FD_F+I & = & \left(
\begin{array}{cc}
B_{11} & B_{12}\\
B_{21} & B_{22}
\end{array}
\right),\\
L_HD_H+I & = & \left(
\begin{array}{cc}
C_{11} & C_{12}\\
C_{21} & C_{22}
\end{array}
\right),\\
LD+I & = & \left(
\begin{array}{ccc}
A_{11} & A_{12} & A_{13} \\
A_{21} & A_{22} & A_{23} \\
A_{31} & A_{32} & A_{33}
\end{array}
\right).
\end{eqnarray*}


Now we calculate the entries of $LD+I$. For the entries of the row corresponding to the vertex $x_0$, we have
\begin{eqnarray*}
A_{11}&=&f^Tc+g^Td+1\\
&=&B_{11}-1+C_{11}-1+1\\
&=&(\beta_F)_{x_0}+(\beta_H)_{x_0}-1=\beta_{x_0},
\end{eqnarray*}
\begin{eqnarray*}
A_{12}&=&\pi a^T+f^TD_{F,x_0}+g^TN\\
&=&\pi a^T+B_{12}-\varphi a^T+g^Td\j^T+g^T\j a^T\\
&=&\pi a^T+B_{12}-\varphi a^T+(C_{11}-1)\j^T-\tau a^T\\
&=&[\pi -\varphi -\tau]a^T+(\beta_F)_{x_0}\j^T+((\beta_H)_{x_0}-1)\j^T\\
&=&[(\beta_F)_{x_0}+(\beta_H)_{x_0}-1]\j^T=\beta_{x_0}\j^T,
\end{eqnarray*}
\begin{eqnarray*}
A_{13}&=&\pi b^T+f^TM+g^TD_{H,x_0}\\
&=&\pi b^T+f^Tc\j^T+f^T\j b^T+C_{12}-\tau b^T\\
&=&\pi b^T+(B_{11}-1)\j^T-\varphi b^T+C_{12}-\tau b^T\\
&=&\pi b^T+((\beta_F)_{x_0}-1)\j^T-\varphi b^T+(\beta_H)_{x_0}\j^T-\tau b^T\\
&=&[\pi -\varphi -\tau]b^T+[(\beta_F)_{x_0}+(\beta_H)_{x_0}-1]\j^T\\
&=&\beta_{x_0}\j^T.
\end{eqnarray*}
For the entries of the rows corresponding to the vertices of $F-x_0$, we have
\begin{eqnarray*}
A_{21}&=&L_{F,x_0}c=B_{21}=\hat{\beta}_F,\\
A_{22}&=&ha^T+L_{F,x_0}D_{F,x_0}+I=B_{22}=\hat{\beta}_F\j^T,\\
A_{23}&=&hb^T+L_{F,x_0}M\\
&=&hb^T+L_{F,x_0}c\j^T+L_{F,x_0}\j b^T\\
&=&hb^T+B_{21}\j^T-hb^T\\
&=&\hat{\beta}_F\j^T.
\end{eqnarray*}
For the entries of the rows corresponding to the vertices of $H-x_0$, we have
\begin{eqnarray*}
A_{31}&=&L_{H,x_0}d=C_{21}=\hat{\beta}_H,\\
A_{32}&=&ka^T+L_{H,x_0}N\\
&=&ka^T+L_{H,x_0}d\j^T+L_{H,x_0}\j a^T\\
&=&ka^T+C_{21}\j^T-ka^T\\
&=&\hat{\beta}_H\j^T,\\
A_{33}&=&kb^T+L_{H,x_0}D_{H,x_0}+I=C_{22}=\hat{\beta}_H\j^T.
\end{eqnarray*}
So we get that $LD+I=\beta\j^T$.
\end{proof}

\begin{remark}
When $\alpha_i=\beta_i$ for each $1\leqslant i\leqslant r$, the above theorem implies results in~\cite{Zhou2017DMdistance-wellDefined}. So all the graphs there can be dealt with by the above theorem. Besides, we give new applications of the above theorem to weighted cactoid digraphs in the following section.
\end{remark}

\begin{remark}\label{remark D does not depent on D_i}
In Theorem~\ref{thm inverse of distance matrix by modified LapExp matrix}, for some index $1\leqslant i\leqslant r$, the case $\lambda_i$ is not invertible may happen while $\lambda$ is invertible. This means the matrix $D$ is invertible doesn't depend on whether each $D_i$ is invertible or not, but on the sum $\sum\limits_{i=1}^r\lambda_i$ is invertible.
\end{remark}



\section{The weighted cactoid digraph}\label{section weighted cactoid digraph}

An edge with a direction is called an \emph{arc}. A \emph{directed path} is a sequence of arcs with the same direction which is from each vertex to its successor in the sequence. A \emph{directed cycle} is a directed path with its starting vertex and ending vertex coincide with each other. Each path (or cycle) in an undirected graph corresponds to two directed paths (or directed cycles) with opposite directions. Note that a directed cycle can have length two, but the minimum length of a cycle is three. A \emph{cactoid digraph} is a strongly connected directed graph whose blocks are directed cycles. The inverse of the distance matrix of a cactoid digraph was studied in~\cite{Hou Chen Inverse of cactoid graph}. Here we study the \emph{weighted cactoid digraph} which is a strongly connected directed graph with each of its block a weighted directed cycle. The weighted directed cycle $\w C_n$ is defined below.

Let $\mathbb{Z}$ be the ring of integers. Let $n\geqslant 2$, and let $\mathbb{Z}_n$ be the ring of integers modulo $n$. Let $\phi$ be the natural map from $\mathbb{Z}$ to $\mathbb{Z}_n$, and let \begin{center}$X=\{0,1,\ldots,n-1\}\subseteq \mathbb{Z}$.\end{center} Then $\mathbb{Z}_n=\phi(X)=\{\phi(0),\phi(1),\ldots,\phi(n-1)\}$. Let $\d C_n$ be a directed cycle with vertex set $V(\d C_n)=\mathbb{Z}_n$ and arc set $E(\d C_n)=\{(\phi(i),\phi(i+1))\mid i\in X\}$. So for any $i\in X$ and $k\in \mathbb{Z}$, the vertex $\phi(i+kn)$ coincides with $\phi(i)$. Then the distance matrix of $\d C_n$ is
\begin{eqnarray*}
\D(\d C_n)=(\partial_{\phi(i),\phi(j)})_{i,j\in X} & = & P+2P^2+\cdots +(n-1)P^{n-1}\\
& = & \left(\begin{array}{cccccc}
      0 & 1 & 2 & \cdots & n-2 & n-1\\
      n-1 & 0 & 1 & \cdots & n-3 & n-2\\
      n-2 & n-1 & 0 & \cdots & n-4 & n-3\\
      \vdots & \vdots & \vdots & \ddots & \vdots & \vdots\\
      2 & 3 & 4 & \cdots & 0 & 1\\
      1 & 2 & 3 & \cdots & n-1 & 0
      \end{array}\right),
\end{eqnarray*}
where $P=(p_{\phi(i),\phi(j)})_{i,j\in X}$ is the cyclic permutation matrix whose nonzero entries are $p_{\phi(i),\phi(i+1)}=1$ ($i\in X$). Notice that the distance satisfying
\begin{equation*}
\partial_{\phi(i),\phi(j)}\in X\text{ and }\partial_{\phi(i),\phi(j)}\equiv j-i \pmod{n}\text{~~~~for $i,j\in X$}.
\end{equation*}


Let $\mathbb{R}$ be a commutative ring with identity. We use $\frac{1}{r}$ to denote the multiplicative inverse of an invertible element $r\in \mathbb{R}$. The \emph{weighted directed cycle} $\w C_n$ is obtained from $\d C_n$ by giving each arc $(\phi(i),\phi(i+1))$ a weight $w_i\in \mathbb{R}$, where $i\in X$. Then the distance matrix of $\w C_n$ is $\D(\w C_n)=(d_{\phi(i),\phi(j)})_{i,j\in X}$, where
\begin{equation*}
d_{\phi(i),\phi(j)}=\left\{\begin{array}{ll}
0, & \text{if $i=j$};\\
\\
\sum\limits_{k=0}^{\partial_{\phi(i),\phi(j)}-1}w_{\phi(i+k)}, & \text{if $i\neq j$}.
\end{array}\right.
\end{equation*}
Let \begin{eqnarray}
w & = & \sum\limits_{i=0}^{n-1}w_{\phi(i)},\\
w^{(2)} & = & \sum\limits_{0\leqslant i<j\leqslant n-1}w_{\phi(i)}w_{\phi(j)}.
\end{eqnarray}
We call $w$ and $w^{(2)}$ \emph{the first weight} and \emph{the second weight} of $\w C_n$, respectively.

Let $A=(a_{ij})$ be an $n\times n$ matrix. We denote its \emph{determinant} by $\det(A)$ and $A(i|j)$ the \emph{submatrix} obtained from $A$ by deleting the $i$-th row and $j$-th column. The \emph{cofactor} of the entry $a_{ij}$ is $(-1)^{i+j}\det(A(i|j))$. The \emph{cofactor} $\cof(A)$ of $A$ is defined as the sum of cofactors of the entries of $A$, i.e. $\cof(A)=\sum\limits_{i,j=1}^n (-1)^{i+j}\det(A(i|j))$. We denote the \emph{adjoint matrix} of $A$ by $\adj(A)=(b_{ij})_{n\times n}$, i.e. $b_{ji}=(-1)^{i+j}\det(A(i|j))$. By the definitions, we have $\cof(A)=\j^T \adj(A) \j$. If $A$ is invertible, then $\displaystyle A^{-1}=\frac{\adj(A)}{\det(A)}$, and so we have \begin{equation}\label{eqn cof=det jT A-1 j}\cof(A)=\det(A)\j^TA^{-1}\j.\end{equation}

\begin{lemma}\label{lem weighted directed cycle}
We use the above notation, and let $W=\D(\w C_n)$. Suppose the first weight $w$ is invertible, and let
\begin{eqnarray*}
\lambda & = & \frac{w^{(2)}}{w},\\
\alpha & = & \frac{1}{w}[w_{\phi(n-1)},w_{\phi(0)},\ldots,w_{\phi(n-2)}]^T,\\
\beta & = & \frac{1}{w}[w_{\phi(0)},w_{\phi(1)},\ldots,w_{\phi(n-1)}]^T,\\
L & = & \frac{1}{w}(I-P).
\end{eqnarray*}
Then
\begin{enumerate}
\item $\alpha^T\j=1$, $L\j=\mathbf{0}$, $\alpha^TW=\lambda\j^T$ and $LW+I=\beta\j^T$; and
\item $\j^T\beta=1$, $\j^TL=\mathbf{0}$, $W\beta=\lambda\j$ and $WL+I=\j\alpha^T$.
\end{enumerate}
\end{lemma}

\begin{proof}
The proofs of the two results are similar, so we only give the proof of the first one, i.e. we will show \begin{center}$\alpha^T\j=1$, $L\j=\mathbf{0}$, $\alpha^TW=\lambda\j^T$ and $LW+I=\beta\j^T$.\end{center} The conditions $\alpha^T\j=1$ and $L\j=\mathbf{0}$ are obvious. Now we show $\alpha^TW=\lambda\j^T$. Let $W_{\phi(0)},W_{\phi(1)},\ldots,W_{\phi(n-1)}$ be all the column vectors of $W$, i.e. \begin{center}$W=[W_{\phi(0)},W_{\phi(1)},\ldots,W_{\phi(n-1)}]$.\end{center} For $j\in X$, we have

\begin{eqnarray*}
w\alpha^TW_{\phi(j)} & = & \sum\limits_{i=0}^{n-1}w_{\phi(i)}d_{\phi(i+1),\phi(j)}\\
 & = & \sum\limits_{i=0}^{j-2}w_{\phi(i)}d_{\phi(i+1),\phi(j)}+\sum\limits_{i=j}^{n-1}w_{\phi(i)}d_{\phi(i+1),\phi(j)}\\
 & = & \sum\limits_{i=0}^{j-2}w_{\phi(i)}\sum\limits_{k=i+1}^{j-1}w_{\phi(k)}
+\sum\limits_{i=j}^{n-1}w_{\phi(i)}\sum\limits_{k=i+1}^{n+j-1}w_{\phi(k)}\\
 & = & \sum\limits_{i=0}^{j-2}w_{\phi(i)}\sum\limits_{k=i+1}^{j-1}w_{\phi(k)}
+\sum\limits_{i=j}^{n-1}w_{\phi(i)}\left(\sum\limits_{k=i+1}^{n-1}w_{\phi(k)}+\sum\limits_{k=n}^{n+j-1}w_{\phi(k)}\right)\\
 & = & \sum\limits_{i=0}^{j-1}w_{\phi(i)}\sum\limits_{k=i+1}^{j-1}w_{\phi(k)}
+\sum\limits_{i=0}^{j-1}w_{\phi(i)}\sum\limits_{k=j}^{n-1}w_{\phi(k)}
+\sum\limits_{i=j}^{n-1}w_{\phi(i)}\sum\limits_{k=i+1}^{n-1}w_{\phi(k)}\\
 & = & \sum\limits_{i=0}^{j-1}w_{\phi(i)}\sum\limits_{k=i+1}^{n-1}w_{\phi(k)}
+\sum\limits_{i=j}^{n-1}w_{\phi(i)}\sum\limits_{k=i+1}^{n-1}w_{\phi(k)}\\
 & = & \sum\limits_{i=0}^{n-1}w_{\phi(i)}\sum\limits_{k=i+1}^{n-1}w_{\phi(k)}\\
 & = & w^{(2)}.
\end{eqnarray*}
Hence we have $\alpha^TW=\lambda\j^T$.

At last we calculate $Y=LW+I=(y_{\phi(i),\phi(j)})_{i,j\in X}$. Let $i,j\in X$ with $i\neq j$. Then
\begin{eqnarray*}
wy_{\phi(i),\phi(i)} & = & d_{\phi(i),\phi(i)}-d_{\phi(i+1),\phi(i)}+w\\
 & = & w-d_{\phi(i+1),\phi(i)}=d_{\phi(i),\phi(i+1)}=w_{\phi(i)},\\
wy_{\phi(i),\phi(j)} & = & d_{\phi(i),\phi(j)}-d_{\phi(i+1),\phi(j)}=w_{\phi(i)}
\end{eqnarray*}
This means $wY=w\beta\j^T$. So we have $LW+I=\beta\j^T$.
\end{proof}

\begin{lemma}\label{lem det of dCn}
We use the above notation, and let $W=\D(\w C_n)$. Suppose the first weight $w$ is invertible. Then the determinant of $W$ is
\begin{equation}\label{eqn det weighted directed cycle}
\det(W)=(-1)^{n-1}w^{n-2}w^{(2)}.
\end{equation}
\end{lemma}

\begin{proof}
By elementary determinant evaluations, we get the following matrices $X_1=W$, $X_2$, $X_3=\left(\begin{array}{cc}1 & \j^T\\ \mathbf{0} & X_2\end{array}\right)$, $X_4$ and $X_5$. For $1\leqslant i\leqslant 5$, we use $R_i(j)$ to denote the $j$-th row of $X_i$ and $C_i(j)$ to denote the $j$-th column of $X_i$. The following are the elementary operations.
\begin{enumerate}
\item $X_2$ is obtained from $X_1$ by the following orerations:
\begin{eqnarray*}
R_2(1) & = & R_1(1)-R_1(2),\\
R_2(2) & = & R_1(2)-R_1(3),\\
& \cdots & \\
R_2(n-1) & = & R_1(n-1)-R_1(n),\\
R_2(n) & = & R_1(n).
\end{eqnarray*}
\item $X_4$ is obtained from $X_3$ by the following orerations:
\begin{eqnarray*}
R_4(1) & = & R_3(1),\\
R_4(2) & = & R_3(2)-w_{\phi(0)}R_3(1),\\
R_4(3) & = & R_3(3)-w_{\phi(1)}R_3(1),\\
& \cdots & \\
R_4(n) & = & R_3(n)-w_{\phi(n-2)}R_3(1),\\
R_4(n+1) & = & R_3(n+1).
\end{eqnarray*}
\item $X_5$ is obtained from $X_4$ by the following orerations:
\begin{eqnarray*}
C_5(1) & = & C_4(1)-\sum\limits_{i=0}^{n-2}\frac{w_{\phi(i)}}{w}C_4(i+2),\\
C_5(i) & = & C_4(i), \text{ for $2\leqslant i\leqslant n+1$}.
\end{eqnarray*}
\end{enumerate}

\begin{equation*}
X_1 = \left(\begin{array}{cccccc}
                        0 & d_{\phi(0),\phi(1)} & d_{\phi(0),\phi(2)} & \cdots & d_{\phi(0),\phi(n-2)} & d_{\phi(0),\phi(n-1)}\\
                        d_{\phi(1),\phi(0)} & 0 & d_{\phi(1),\phi(2)} & \cdots & d_{\phi(1),\phi(n-2)} & d_{\phi(1),\phi(n-1)}\\
                        d_{\phi(2),\phi(0)} & d_{\phi(2),\phi(1)} & 0 & \cdots & d_{\phi(2),\phi(n-2)} & d_{\phi(2),\phi(n-1)}\\                        \vdots & \vdots & \vdots & \ddots & \vdots & \vdots\\
                        d_{\phi(n-2),\phi(0)} & d_{\phi(n-2),\phi(1)} & d_{\phi(n-2),\phi(2)} & \cdots & 0 & d_{\phi(n-2),\phi(n-1)}\\                        d_{\phi(n-1),\phi(0)} & d_{\phi(n-1),\phi(1)} & d_{\phi(n-1),\phi(2)} & \cdots & d_{\phi(n-1),\phi(n-2)} & 0
                        \end{array}\right),
\end{equation*}
\begin{equation*}
X_2 = \left(\begin{array}{cccccc}
                        w_{\phi(0)}-w & w_{\phi(0)} & w_{\phi(0)} & \cdots & w_{\phi(0)} & w_{\phi(0)}\\
                        w_{\phi(1)} & w_{\phi(1)}-w & w_{\phi(1)} & \cdots & w_{\phi(1)} & w_{\phi(1)}\\
                        w_{\phi(2)} & w_{\phi(2)} & w_{\phi(2)}-w & \cdots & w_{\phi(2)} & w_{\phi(2)}\\
                        \vdots & \vdots & \vdots & \ddots & \vdots & \vdots\\
                        w_{\phi(n-2)} & w_{\phi(n-2)} & w_{\phi(n-2)} & \cdots & w_{\phi(n-2)}-w & w_{\phi(n-2)}\\                         d_{\phi(n-1),\phi(0)} & d_{\phi(n-1),\phi(1)} & d_{\phi(n-1),\phi(2)} & \cdots & d_{\phi(n-1),\phi(n-2)} & 0
                        \end{array}\right),
\end{equation*}
\begin{equation*}
X_4 = \left(\begin{array}{ccccccc}
                        1 & 1 & 1 & 1 & \cdots & 1 & 1\\
                        -w_{\phi(0)} & -w & 0 & 0 & \cdots & 0 & 0\\
                        -w_{\phi(1)} & 0 & -w & 0 & \cdots & 0 & 0\\
                        -w_{\phi(2)} & 0 & 0 & -w & \cdots & 0 & 0\\
                        \vdots & \vdots & \vdots & \vdots & \ddots & \vdots & \vdots\\
                        -w_{\phi(n-2)} & 0 & 0 & 0 & \cdots & -w & 0\\
                        0 & d_{\phi(n-1),\phi(0)} & d_{\phi(n-1),\phi(1)} & d_{\phi(n-1),\phi(2)} & \cdots & d_{\phi(n-1),\phi(n-2)} & 0
                        \end{array}\right),
\end{equation*}
\begin{equation*}
X_5 = \left(\begin{array}{ccccccc}
                        a & 1 & 1 & 1 & \cdots & 1 & 1\\
                        0 & -w & 0 & 0 & \cdots & 0 & 0\\
                        0 & 0 & -w & 0 & \cdots & 0 & 0\\
                        0 & 0 & 0 & -w & \cdots & 0 & 0\\
                        \vdots & \vdots & \vdots & \vdots & \ddots & \vdots & \vdots\\
                        0 & 0 & 0 & 0 & \cdots & -w & 0\\
                        b & d_{\phi(n-1),\phi(0)} & d_{\phi(n-1),\phi(1)} & d_{\phi(n-1),\phi(2)} & \cdots & d_{\phi(n-1),\phi(n-2)} & 0
                        \end{array}\right).
\end{equation*}

In matrix $X_5$, the parameters
\begin{eqnarray*}
a & = & 1-\sum\limits_{i=0}^{n-2}\frac{w_{\phi(i)}}{w}=\frac{w_{\phi(n-1)}}{w},\\
b & = & -\sum\limits_{i=0}^{n-2}\frac{d_{\phi(n-1),\phi(i)}w_{\phi(i)}}{w}=-R_1(n)\beta.
\end{eqnarray*}
By Lemma~\ref{lem weighted directed cycle}, we have $wX_1\beta=w^{(2)}\j$. Then for $1\leqslant j\leqslant n$, we get $wR_1(j)\beta=w^{(2)}$. So $-wb=w^{(2)}$. Hence, $\det(W)=\det(X_2)=\det(X_3)=\det(X_4)=\det(X_5)= -b(-w)^{n-1}=(-1)^{n-1}w^{n-2}w^{(2)}$. Note that, when the ring $\mathbb{R}$ is the real numbers, the formula of the determinant $\det(W)$ is still true even if the first weight $w$ is not invertible.
\end{proof}

By Lemma~\ref{lem weighted directed cycle}, the distance matrix $W=\D(\w C_n)$ of the weighted directed cycle $\w C_n$ is a left and right $\LapExp^*(\lambda,\alpha,\beta,L)$ matrix, where $\lambda$, $\alpha$, $\beta$ and $L$ are defined in Lemma~\ref{lem weighted directed cycle}. So the bag $B=(W,\lambda,\alpha,\beta,L)$ is a left and right $\LapExp^*$ $n$-bag. The bag $B$ is called \emph{the natural bag} of the weighted directed cycle $\w C_n$. Suppose the first weight $w$ is invertible, then the determinant $\det(W)=(-1)^{n-1}w^{n-2}w^{(2)}$ by Lemma~\ref{lem det of dCn}. If we assume the second weight $w^{(2)}$ is invertible additionally, then $\lambda$ is invertible and by Lemmas~\ref{lem inv of a matrix by lap} and \ref{lem Laplacian-like matrix}, $W$ is invertible, \begin{center}$W^{-1}=-L+\frac{1}{\lambda}\beta\alpha^T=-L+\frac{w}{w^{(2)}}\beta\alpha^T$,\end{center} and $L$ is a Laplacian-like matrix. By Equation~(\ref{eqn cof=det jT A-1 j}), we have
\begin{equation}\label{eqn cof weighted directed cycle}
\cof(W)=\det(W)\j^TW^{-1}\j=\frac{1}{\lambda}\det(W)=(-1)^{n-1}w^{n-1}.
\end{equation}


As an application of Theorem~\ref{thm inverse of distance matrix by modified LapExp matrix}, we get the $\LapExp^*$ property of the distance matrix of the weighted cactoid digraph.


\begin{theorem}\label{thm weighted cactoid digraph}
Let $G$ be a weighted cactoid digraph with structure parameters $(G^n;G_1^{n_1},G_2^{n_2},\ldots,G_r^{n_r})$. For each $1\leqslant i\leqslant r$, suppose the first weight $w_i$ of the weighted directed cycle $G_i$ is invertible, and let $B_i=(W_i,\lambda_i,\alpha_i,\beta_i,L_i)$ be the natural bag of $G_i$. Let $B=(W_c,\lambda,\alpha,\beta,L)$ be the composition bag of natural bags $B_1,B_2,\ldots,B_r$, where $W_c$ is the distance matrix of $G$. Then $B$ is a left and right $\LapExp^*$ bag, and $L$ is a Laplacian-like matrix . Furthermore, if $\lambda=\sum\limits_{i=1}^r\lambda_i$ is invertible, then $W_c$ is invertible and $W_c^{-1}=-L+\frac{1}{\lambda}\beta\alpha^T$.
\end{theorem}

Graham, Hoffman and Hosoya~\cite{Graham DM of a directed graph} showed a very attractive theorem, expressing the determinant of the distance matrix of a distance well-defined graph explicitly as a function of its blocks. Applying their result to the weighted cactoid digraph in Theorem~\ref{thm weighted cactoid digraph} and using Equations~(\ref{eqn det weighted directed cycle}) and (\ref{eqn cof weighted directed cycle}), we have
\begin{eqnarray}
\cof(W_c) & = & \prod\limits_{i=1}^r\cof(W_i) = \prod\limits_{i=1}^r(-1)^{n_i-1}w_i^{n_i-1}=(-1)^{n-1}\prod\limits_{i=1}^rw_i^{n_i-1},\label{eqn cof weighted cactoid digraph}\\
\det(W_c) & = & \sum\limits_{i=1}^r\det(W_i)\prod\limits_{j=1,j\neq i}^r\cof(W_j)=\cof(W_c)\sum\limits_{i=1}^r\frac{\det(W_i)}{\cof(W_i)}\nonumber\\
& = & \cof(W_c)\sum\limits_{i=1}^r\lambda_i=\lambda\cof(W_c)=(-1)^{n-1}\lambda\prod\limits_{i=1}^rw_i^{n_i-1}.\label{eqn det weighted cactoid digraph}
\end{eqnarray}
If we take each edge as a directed cycle of length two with each arc weight one, then Equation~(\ref{eqn det weighted cactoid digraph}) implies Graham and Pollak's formula~\cite{Graham Pollak on addressing problem LS} \begin{center}$\det(\D(T_n))=(-1)^{n-1}2^{n-2}(n-1),$\end{center} and Theorem~\ref{thm weighted cactoid digraph} implies the inverse of $\D(T_n)$~\cite{Bapat on DM and Lap,Graham DM Polynomials of trees}, where $\D(T_n)$ is the distance matrix of the tree $T_n$ on $n$ vertices.

Since $w_i$ is invertible for each $1\leqslant i\leqslant r$, the distance matrix $W_c$ is invertible if and only if $\lambda$ is invertible by Equation~(\ref{eqn det weighted cactoid digraph}). The above Theorem~\ref{thm weighted cactoid digraph} implies part of the results in~\cite{Bapat bidirected tree,Hou Chen Inverse of cactoid graph,ZhouDing2016DMweightedTree}.


\section{Classes of distance well-defined graphs}

Recall that a graph is distance well-defined if it is a connected undirected graph, a strongly connected directed graph or a strongly connected mixed graph. So the distance matrix of a distance well-defined graph is well-defined. Here we define several classes of distance well-defined graphs.

A distance well-defined graph $G$ is called distance invertible, if the distance matrix $\D(G)$ is invertible. Let $\DMI$ be the set of distance invertible graphs. A distance well-defined graph $G$ is modified left (or right) Laplacian expressible, if the distance matrix $\D(G)$ is a left (or right) $\LapExp^*(\lambda,\alpha,\beta,L)$ matrix with $\lambda\neq 0$ and $\j^T\beta=1$ (or $\alpha^T\j=1$). A distance well-defined graph $G$ is modified Laplacian expressible, if either it is a modified left Laplacian expressible graph or it is a modified right Laplacian expressible graph. Let $\LapExp^*(\mathcal{L})$ be the set of modified left Laplacian expressible graphs. Let $\LapExp^*(\mathcal{R})$ be the set of modified right Laplacian expressible graphs. Let $\LapExp^*$ be the set of modified Laplacian expressible graphs. The relations between these classes of distance well-defined graphs are as follows, \begin{center}$\LapExp^*(\mathcal{L})\cup\LapExp^*(\mathcal{R})= \LapExp^*\subseteq \DMI$.\end{center}

Let $r\geqslant 2$ and $n_i\geqslant 2$ for each $1\leqslant i\leqslant r$. Let $G$ be the complete multipartite graph $K_{n_1,n_2,\ldots,n_r}$. If we suppose \begin{center}$\#\{i\mid n_i=2, 1\leqslant i \leqslant r\}\leqslant 1$,\end{center} then the distance matrix $\D(G)$ is invertible by Corollary~2.5 in~\cite{Zhou2017DMdistance-wellDefined}; furthermore, by the discussion on page~22 in~\cite{Zhou2017DMdistance-wellDefined}, we know that the distance matrix $\D(G)$ is Laplacian expressible. So we have the following lemma.

\begin{lemma}
Let $r\geqslant 2$ and $n_i\geqslant 2$ for each $1\leqslant i\leqslant r$. Let $G$ be the complete multipartite graph $K_{n_1,n_2,\ldots,n_r}$. If $G\in \DMI$, then $G\in \LapExp^*$.
\end{lemma}

Hence the following question is interesting.
\begin{question}
Is there a distance well-defined graph $G\in \DMI\setminus \LapExp^*$?
\end{question}

Let $\LapExp_*=\LapExp^*(\mathcal{L})\cap\LapExp^*(\mathcal{R})$. By examples in~\cite{ZhouDing2017MixedBlockGraphs,Zhou2017DMdistance-wellDefined} and Section~\ref{section weighted cactoid digraph}, we get that $\LapExp_*\neq\emptyset$ and the following graphs are in $\LapExp_*$: the complete graph $K_m$ where $m\geqslant 2$, the odd cycle $C_{2k+1}$ where $k\geqslant 1$, the directed cycle $dC_t$ where $t\geqslant 2$, the complete multipartite graph $K_{n_1,n_2,\ldots,n_r}$ where $r\geqslant 2$, $n_i\geqslant 2$ for each $1\leqslant i\leqslant r$ and $\#\{i\mid n_i=2, 1\leqslant i \leqslant r\}\leqslant 1$, the mixed complete graph $\m K_s$ where $s\geqslant 3$, and the weighted directed cycle $\w C_n$ where $n\geqslant 2$ and the weights satisfying $\lambda=\frac{w^{(2)}}{w}$ is invertible (see Lemma~\ref{lem weighted directed cycle}). By Remark~\ref{remark D does not depent on D_i}, we have the following result.

\begin{corollary}\label{corollary graphs with known LapExp blocks}
Let $G$ be a distance well-defined graph with structure parameters $(G^n;G_1^{n_1},G_2^{n_2},\ldots,G_r^{n_r})$. Suppose for each $1\leqslant i\leqslant r$, the block $G_i$ is one of the following graphs: the complete graph $K_m$ where $m\geqslant 2$, the odd cycle $C_{2k+1}$ where $k\geqslant 1$, the directed cycle $dC_t$ where $t\geqslant 2$, the complete multipartite graph $K_{n_1,n_2,\ldots,n_r}$ where $r\geqslant 2$ and $n_i\geqslant 2$ for each $1\leqslant i\leqslant r$, the mixed complete graph $\m K_s$ where $s\geqslant 3$, and the weighted directed cycle $\w C_n$ where $n\geqslant 2$. Then the distance matrix of each $G_i$ is a left and right $\LapExp^*(\lambda_i,\alpha_i,\beta_i,L_i)$ matrix (may not invertible). We suppose $\lambda=\sum\limits_{i=1}^r\lambda_i$ is invertible. Then $G\in \LapExp_*$.
\end{corollary}


Corollary~\ref{corollary graphs with known LapExp blocks} implies many known results. We want to find more examples of graphs in $\LapExp_*$. So the following questions are interesting.

\begin{question}
Is there a distance well-defined graph $G$ satisfying either
$G\in \LapExp^*(\mathcal{R})\setminus \LapExp^*(\mathcal{L})$, or $G\in \LapExp^*(\mathcal{L})\setminus \LapExp^*(\mathcal{R})$?
\end{question}

\begin{question}
How to characterize graphs in $\LapExp_*$, $\LapExp^*(\mathcal{L})$, $\LapExp^*(\mathcal{R})$, or $\LapExp^*$?
\end{question}


\section{Acknowledgement}

This work is supported by a project funded by China Postdoctoral Science Foundation under file number 2017M620491.

%



\end{document}